\documentclass[a4paper,reqno]{amsart}
\usepackage{amssymb, amsmath, amscd}
\usepackage{color}
\date{26Apr15}

\numberwithin{equation}{section}
\newtheorem{thm}{Theorem}

\newtheorem{lem}[thm]{Lemma}

\newtheorem{E2.a}[thm]{Example}

\def\EL{\operatorname{EL}}

\def\thesection{%
\arabic{section}} 

\begin{document}
\makeatletter
\renewcommand{\theequation}{%
\thesection.\alph{equation}} \@addtoreset{equation}{section}
\makeatother
\title[Euler-Lagrange]{Euler-Lagrange formulas for pseudo-K\"ahler manifolds}
\author{JeongHyeong Park}
\address{J. H. Park :Department of Mathematics,
Sungkyunkwan University, 2066 Seobu-ro Suwon, 440-746, Korea}
\email{parkj@skku.edu}
\begin{abstract}
 Let $c$ be a characteristic form of degree $k$ which is defined
on a K\"ahler manifold of real dimension $m>2k$. Taking the inner
product with the K\"ahler form $\Omega^k$ gives a scalar invariant
which can be considered as a generalized Lovelock functional. The
associated Euler-Lagrange equations are a generalized
Einstein-Gauss-Bonnet gravity theory; this theory restricts to the
canonical formalism if $c=c_2$ is the second Chern form. We extend
previous work studying these equations from the K\"ahler to the
pseudo-K\"ahler setting.\end{abstract} \maketitle 
\section{Introduction}
\subsection{Historical perspective}
The Euler form was introduced by Chern \cite{C1} in his generalization
of the Gauss-Bonnet theorem to higher dimensions. Let $R_{ijkl}$ be
the components of the Riemann curvature tensor. Let $m=2\bar  m$
be even. The Pfaffian or Euler integrand is defined to be
$$
E_m(g)=\frac1{(8\pi)^{\bar  m}{\bar  m}!}R_{i_1i_2j_1j_2}\dots R_{i_{m-1}i_mj_mj_{m-1}}
g(e^{i_1}\wedge\dots\wedge e^{i_m},e^{j_1}\wedge\dots\wedge e^{j_m})
$$
where we adopt the {\it Einstein convention} and sum over repeated indices. Let
$\operatorname{dvol}_g:=\sqrt{\det(g_{ij})}dx^1\dots dx^m$ be the Riemannian measure.
Chern showed that if $M$ is a compact Riemannian manifold of dimension $m$ without boundary,  then
the Euler-Poincare characteristic $\chi(M)$ is given in terms of curvature:
$$
\chi(M^m)=\int_ME_m(g)\operatorname{dvol}_g\,.
$$
This generalizes the classical Gauss-Bonnet formula from 2 dimensions to the higher dimensional setting.
Subsequently, Chern \cite{C2} introduced the so called {\it Chern
classes}; these will discussed in more detail in the next section.
Let $c_{\bar  m}$ be the $\bar  m^{\operatorname{th}}$ Chern
class. If $(M,g,J)$ is a K\"ahler manifold, then
$E_m(g)=c_{\bar  m}(g)$ so this particular characteristic class reproduces
the Euler form. The theory of characteristic classes is, of course,
much more general and plays an important role in the
Hirzebruch-Riemann-Roch Theorem \cite{H78} amongst many applications.

Alternative theories of gravity arise naturally in physics since the standard
particle model and general relative seem to fail
at extreme regimes of ultra violet scales.
Lovelock \cite{L} introduced Chern-Gauss-Bonnet gravity by studying
the Euler-Lagrange equations associated with
$ E_4=\frac1{32\pi^2}\{\tau^2-4|\rho|^2+|R|^2\}$ in order to
study to the Einstein field equations in vacuo; it is a non-linear theory of gravity. The action is crucial
in dimensions $m>4$. These equations (with appropriate perturbing terms) have
been used by D. Chirkov, S. Pavluchenko and A.
Toporensky  \cite{CP, CP1} to investigate the constant volume
exponential solutions in the Einstein-Gauss-Bonnet gravity in $4+1$ and $5+1$ dimensional
space times; the metric is, of course, Lorentzian. Work by
Najian \cite{Na14} treats aspects of holographic dual of boundary conformal field theories
for higher derivative Gauss-Bonnet (GB) gravity. Nozari et al. \cite{N13} treat a DGP-inspired braneworld model
such that the induced gravity on the brane uses a bulk action which
contains the Gauss-Bonnet term to incorporate higher order curvature effects.
Zeng and Liu \cite{ZL13} study the thermalization of a dual conformal field theory to
Gauss-Bonnet gravity by modeling a thin-shell of dust that interpolates between
a pure AdS and a Gauss-Bonnet AdS black brane.

 The Euler-Lagrange equations have also been used \cite{J} by J. B. Jimenez and T. S. Koivisto in
the context of Weyl geometry to investigate an extended
Gauss-Bonnet gravity theory in arbitrary dimensions and in a space
provided with a Weyl connection. See also work of D. Butter et al. \cite{B13}
 on the Gauss-Bonnet density in superspace and work of W. Yao and J. Jing \cite{YJ13}
 dealing with a Born-Infeld electromagnetic field coupling with a charged scalar field in the five-dimensional
 Einstein-Gauss-Bonnet spacetime. The field is a vast one and we can not do justice to it with this brief summary.

The higher dimensional Euler integrands are important; for example
Dadhich and Pons \cite{DP13} examine Lovelock
 functional $\mathcal{L}_{2k}$ for arbitrary $k$ in Lorentzian signature.
  Although a-priori the Euler-Lagrange
equations for the Lovelock functional defined by $E_m$ can involve
the $4^{\operatorname{th}}$ derivatives of the metric in dimensions $n>m$,
Berger \cite{M70} conjectured it only involved
curvature; this was subsequently verified by Kuz'mina \cite{K74} and
Labbi \cite{L05, L07, L08}.
Following de Lima and de Santos \cite{LS10}, one says that a compact
Riemannian $n$-manifold is $2k$-Einstein
for $2\le 2k<n$ if it is a critical metric for the Einstein-Hilbert-Lovelock functional
$L_{2k}(g)=\int_ME_{2k}\operatorname{dvol}$ when restricted to metrics on $M$ with unit volume.
This involves, of course, examining the associated Euler-Lagrange equations for this functional.
Indefinite signatures are also important; Gilkey, Park and Sekigawa examined the Euler-Lagrange equations
in the higher signature setting \cite{GPS12}.

In previous work \cite{GPS14}, we extended these results for the Lovelock
functional to the setting of characteristic classes in the K\"ahler
context; in higher dimensions it seemed possible that appropriate gravity theories
could be based on arbitrary characteristic classes and not just on the Lovelock functional.
We gave explicit formulas for the appropriate Euler-Lagrange equations and showed that
the map from the characteristic forms to the symmetric 2-tensors
given by the Euler-Lagrange equations coincides with the map given
algebraically by the transgression in the K\"ahler setting -- see Theorem~\ref{T1} below for details.
It is the purpose of this present paper to extend these results to
the indefinite setting with a minimum of technical fuss and in particular not
to repeat the analysis of \cite{GPS14} but rather to use analytic continuation as indefinite
signatures playing a crucial role in many applications; there are string theories where the hidden
dimensions also have indefinite signatures.

\subsection{A review of Chern-Weil Theory and the characteristic classes}
Let $V$ be a real vector bundle of dimension $2\ell$ which is
equipped with an almost complex structure $J$. We use $J$ to
give $V$ the structure of a complex vector bundle $V_c$ by defining
$\sqrt{-1}v:=Jv$. Let $\nabla$ be a connection on $V$ which commutes
with $J$. Since $J$ then commutes with the curvature $R$ of
$\nabla$, we may regard $R$ as a complex 2-form valued
endomorphism $R_c$ of $V_c$. Let $\mathfrak{C}_{k,\ell}$ be
the collection of polynomial maps $\Theta(\cdot)$ from the space of
$\ell\times\ell$ complex matrices $M_\ell(\mathbb{C})$ to
$\mathbb{C}$ which are homogeneous of degree $k$ and which satisfy
$\Theta(gAg^{-1})=\Theta(A)$ for all $A\in M_\ell(\mathbb{C})$ and
all $g$ in the general linear group $GL_\ell(\mathbb{C})$. If
$\Theta\in\mathfrak{C}_{k,\ell}$, we may define
$\Theta(R_c)\in\Lambda^{2k}(M)\otimes_{\mathbb{R}}\mathbb{C}$ invariantly (i.e. independent of the
particular local frame chosen for $V_c$). One has that $\Theta(R_c)$ is a
closed $2k$-form and the de Rham cohomology class of $\Theta(R_c)$
is independent of the particular connection chosen; these are
the celebrated {\it characteristic classes} of Chern \cite{C2}. Other structure groups,
of course, give rise appropriate characteristic classes; the Pontrjagin classes, for example,
relate to the orthogonal group while the Euler form $E_m$ can properly be regarded as an
characteristic class of the special orthogonal group.

\subsection{Holomorphic Geometry}
Let $M$ be a smooth manifold of (real) dimension $m:=2\bar  m$. We say that $J$ is an {\it integrable complex
structure} on the tangent bundle $TM$ if there is a coordinate atlas
with local coordinates $(x^1,\dots,x^m)$ so that
\begin{equation}\label{E1.a}
J\partial_{x_i}=\partial_{x_{i+\bar m}}\quad\text{and} \quad
J\partial_{x_{i+\bar m}}= -\partial_{x_i}\quad\text{for}\quad 1\le
i\le\bar m\,.
\end{equation}
Since $J^2=-\operatorname{id}$, we may use $J$ to
give $TM$ a complex structure and regard $TM$ as a complex vector bundle.
We say that a pseudo-Riemannian metric $h$ on $TM$ is {\it pseudo-Hermitian} if
$J^*h=h$. Let $\nabla^h$ be the Levi-Civita connection of $h$.
Since $\nabla^h$ need not commute with $J$, we form the connection
$$
\tilde\nabla^h:=\textstyle\frac12\{\nabla^h+J^*\nabla^h\}=\textstyle\frac12\{\nabla^h-J\nabla^hJ\}\,.
$$
Let $R_c^h$ be the associated complex curvature tensor. If $\Theta\in\mathfrak{C}_{k,\ell}$, then we may form
$$
\Theta(h):=\Theta(R_c^h)\in\Lambda^{2k}(M)\otimes_{\mathbb{R}}\mathbb{C}\,.
$$

Of particular interest is the special case where the Levi-Civita connection actually does commute with $J$, i.e.
$\nabla^h(J)=0$ and the triple $(M,h,J)$ is said to be a {\it pseudo-K\"ahler} manifold (if $h$ is positive definite, then
$(M,h,J)$ is said to be a {\it K\"ahler} manifold). The pseudo-K\"ahler
geometries are very special as we shall see presently. One feature is that there exist normal-holomorphic
coordinates, i.e.  holomorphic coordinate systems where the first derivatives of the metric vanish. We refer to Section~\ref{S2}
for details.
If $(M,h,J)$ is a pseudo-K\"ahler manifold, then
$\tilde\nabla^h=\nabla^h$. Let $c_\ell(A):=\det(\frac{\sqrt{-1}}{2\pi}A)$;
then $c_\ell\in\mathfrak{C}_{\ell,\ell}$ and $c_\ell(h)=E_{2\ell}\operatorname{dvol}$ is the
integrand of the Chern-Gauss-Bonnet theorem.
\subsection{Euler-Lagrange equations associated to the characteristic classes}
Let $\Omega_h(x,y):=h(x,Jy)$. We assume $(M,h,J)$ is pseudo-K\"ahler; this implies $d\Omega_h=0$.
Any complex manifold inherits a natural orientation so we may identify measures with $m$-forms.
In the positive definite setting, we have under this identification that
$\operatorname{dvol}={\textstyle\frac1{\tilde m!}}\Omega_h^{\tilde m}$
and we replace $\operatorname{dvol}$ by ${\textstyle\frac1{\tilde m!}}\Omega_h^{\tilde m}$ to avoid complications
with signs henceforth in the higher signature context.

Let $\Theta\in\mathfrak{C}_{k,\ell}$. As noted above, $\Theta(h)$ is a differential form of degree $2k$.
To obtain a scalar invariant, we contract with the K\"ahler form $\Omega_h(x,y):=h(x,Jy)$
and in analogy with Gauss-Bonnet gravity (as discussed above),
we consider the scalar invariant $h(\Theta(h),\Omega_h^k)$ and associated Lovelock functional
$$\Theta[M,h,J]:={\textstyle\frac1{\tilde m!}}\int_Mh(\Theta(h),\Omega_h^k)\cdot\Omega_h^{\tilde m}\,.$$
The associated Euler-Lagrange equations are defined by setting:
$$\EL_\Theta(h,\kappa):={\textstyle\frac1{\tilde m!}}\left.\left\{\partial_{\epsilon}\int_M
h(\Theta(h+\epsilon\kappa,\Omega_{h+\epsilon\kappa}^k)
\cdot\Omega_{h+\epsilon\kappa}^{\tilde m}\right\}\right|_{\epsilon=0}$$
where $\kappa$ is a $J$-invariant symmetric $2$-tensor with compact support.
If $\bar  m=k$ and if $M$, then $\Theta[M,h,J]$ is independent of $h$ and hence the associated
Euler-Lagrange equations vanish. We therefore suppose that $k<\bar  m$.
We can integrate by parts to express
$$
\EL_\Theta(h,\kappa)=
{\textstyle\frac1{\tilde m!}}\int_M\langle\mathcal{E}_\Theta(h),\kappa\rangle\cdot\Omega_h^{\tilde m}
$$
where $\mathcal{E}_\Theta(h)\in S^2(TM)$ is a $J$-invariant
symmetric 2-tensor field and where $\langle\cdot,\cdot\rangle$
denotes the natural pairing between $S^2(TM)$ and $S^2(T^*M)$.
Let
$$
\mathcal{F}_\Theta(h):=\textstyle
-\frac1{k+1}\sqrt{-1}h((\Theta(R_c^h)\wedge e^\alpha\wedge e^{\bar\beta})
\otimes e_\alpha\otimes e_{\bar\beta}\,.
$$

Of course, for a general $h$, $\mathcal{E}_\Theta(h)$ is very complicated and can
not be expressed directly in terms of curvature. However, after a lengthy
and difficult calculation in invariance theory, it was shown \cite{GPS14} that
\begin{thm}\label{T1}
Let $\Theta\in\mathfrak{C}_{k,\bar  m}$. If $(M,h,J)$ is a K\"ahler
manifold, then
\smallbreak\centerline{$\mathcal{E}_\Theta(h)=\mathcal{F}_\Theta{{(h)}}$.}
\end{thm}
The main result of this brief
note extends Theorem~\ref{T1} to the pseudo-K\"ahler setting.
\begin{thm}\label{T2}
Let $\Theta\in\mathfrak{C}_{k,\bar  m}$. If $(M,h,J)$ is a
pseudo-K\"ahler manifold, then
\smallbreak\centerline{$\mathcal{E}_\Theta(h)=\mathcal{F}_\Theta{{(h)}}$.}

\end{thm}

One could redo the analysis of \cite{GPS14} taking into account the fact that the structure group $U(p,q)$
involves not only rotations but also hyperbolic boosts. But instead, we will use analytic continuation
to pass from the positive definite to the indefinite setting. Such methods of analytic
continuation have been previously; see, for example, the discussion in Garc\'{i}a-R\'{i}o et al. \cite{GKVV99} that
spacelike and timelike Osserman are equivalent concepts.
In Section~\ref{S2}, we will review some results concerning pseudo-K\"ahler geometry which we shall need
and we shall give the proof of Theorem~\ref{T2} using analytic continuation.

\section{Pseudo-K\"ahler geometry}\label{S2}
Let $(M,h,J)$ be a pseudo-K\"ahler manifold of signature $(2\bar p,2\bar q)$.
Fix a point $P\in M$ and choose
 local coordinates $x=(x^1,\dots,x^m)$ centered at $P$ so that $J$ is given by Equation~(\ref{E1.a}).
 We complexity and set, for $1\le\alpha\le\bar m$,
 $$\begin{array}{ll}
\partial_{z^\alpha}:=\frac12\{\partial_{x^\alpha}-\sqrt{-1}\partial_{x^{\bar m+\alpha}}\},&
dz^\alpha:=\{dx^\alpha+\sqrt{-1}dx^{\bar m+\alpha}\},\\[0.03in]
\partial_{\bar z^\alpha}:=\frac12\{\partial_{x^\alpha}+\sqrt{-1}\partial_{x^{\bar m+\alpha}}\},&
d\bar z^\alpha:=\{dx^\alpha-\sqrt{-1}dx^{\bar m+\alpha}\}.
\end{array}$$
We extend the metric $h$ to be complex bilinear and set
$h_{\alpha,\bar\beta}:=h(\partial_{z^\alpha},\partial_{\bar z^\beta})$. The condition that $J^*h=h$ is then equivalent
to the identities:
 The condition that $J^*h=h$ then is reflected by the identities
$$
h(\partial_{z^\alpha},\partial_{z^\beta})=0,\quad h(\partial_{\bar z^\alpha},\partial_{\bar z^\beta})=0,\quad
\bar h(\partial_{z^\alpha},\partial_{\bar z^\beta})=h(\partial_{z^\beta},\partial_{\bar z^\alpha})\,.
$$
We set $h_{\alpha,\bar\beta}:=h(\partial_{z^\alpha},\partial_{\bar z^\beta})$. We then have $\bar h_{\alpha,\bar\beta}=h_{\beta,\bar\alpha}$.
If we set $h_{\alpha,\bar\beta/\gamma}:=\partial_{z^\gamma}h_{\alpha,\bar\beta}$ and
$h_{\alpha,\bar\beta/\bar\gamma}:=\partial_{\bar z^\gamma}h_{\alpha,\bar\beta}$, the K\"ahler condition becomes:
\begin{equation}\label{E2.a}
h_{\alpha,\bar\beta/\gamma}=h_{\gamma,\bar\beta/\alpha}\quad\text{and}\quad
    h_{\alpha,\bar\beta/\bar\gamma}=h_{\alpha,\bar\gamma/\bar\beta}\,.
\end{equation}
Let $A:=(\alpha_1,\dots,\alpha_\nu)$ and
$B:=(\beta_1,\dots,\beta_\mu)$ are a collection of indices between
$1$ and $\bar m$, we define
$$
h(A;B):=\partial_{z^{\alpha_2}}\dots\partial_{z^{\alpha_\nu}}\partial_{\bar z^{\beta_2}}\dots\partial_{\bar z^{\beta_\mu}}h_{\alpha_1,\bar\beta_1}\,.
$$
It is immediate that $\bar h(A;B)=h(B;A)$. We differentiate Equation~(\ref{E2.a}) to see that we can permute the elements of $A$
and that we can permute the elements of $B$ without changing $h(A;B)$. The following Lemma was proved in \cite{GPS14}
in the positive definite setting. The proof, however, involved quadratic and higher order holomorphic changes and was
independent of the signature of the metric. It extends without change to the setting at hand.

\begin{lem}\label{lem-3}
 Let $P$ be a point of a K\"ahler manifold $M =(M, J, h)$. Fix $n\ge2$.
There exist
local holomorphic coordinates $\vec x=(x^1,\dots,x^{ 2m})$
 centered at $P$ so that
 \begin{enumerate}
 \item $J$ is given by Equation~(\ref{E1.a}).
\item $h(A;B)(P)=0$ for $|B|=1$ and $2\le |A|\le n$.
\end{enumerate}
Fix $n\ge2$. Let constants $c(A;B)\in\mathbb{C}$ be given
for $2\le|A|\le n$ and $2\le|B|\le n$ so that $c(A;B)=\bar c(B;A)$.
There a K\"ahler metric $\tilde h$
on $(M,J)$ so that $\tilde h(A;B)=0$ for $|B|=1$
and so that $\tilde h(A;B)(P)=c(A;B)$
for $2\le |A|\le n$ and $2\le |B|\le n$.
\end{lem}

The variables $\{h(A;B)\}$ are a good choice of variables since,
unlike the covariant derivatives of the curvature tensor, there are
no additional identities and we are
dealing with a pure polynomial algebra.
\begin{proof}[Proof of Theorem~\ref{T2}]
Fix $\Theta\in\mathfrak{C}_{k,\bar m}$. We work purely formally. We
use Lemma~\ref{lem-3} to regard $\mathcal{E}_\Theta$ and
${{\mathcal{F}_\Theta}}$ as polynomials in $h(P)$, the variables
$h(A;B)$, and $\det(h)^{-1}$ where $h$ is a $J$-invariant symmetric
bilinear form on $\mathbb{R}^m$; we must introduce the variable
$\det(h)^{-1}$ to define the metric on the cotangent bundle and to
raise and lower indices. These polynomials are well defined if
$\det(h)\ne0$ and we have $\mathcal{E}_\Theta-\mathcal{F}_\Theta=0$
if $h$ is positive definite.

But, of course, we can allow $h$ to be complex valued.
We have $\mathcal{E}_\Theta-\mathcal{F}_\Theta=0$ if $h$ is real
valued and positive definite. Imposing the condition $\det(h)\ne0$ does
not disconnect the parameter space and thus the identity theorem
yields $\mathcal{E}_\Theta-\mathcal{F}_\Theta=0$ in complete
generality and, in particular, if $h$ has indefinite signature.
\end{proof}

\section*{Acknowledgements}
\noindent This research was supported by Basic Science Research
Program through the National Research Foundation of Korea(NRF)
funded by the Ministry of Education (2014053413). It is a pleasure to
acknowledge helpful conversations with Professor Gilkey concerning
the matters of this paper.

\end{document}